\documentclass[conference, 10pt, twocolumn]{ieeeconf}
\IEEEoverridecommandlockouts
\usepackage{psfrag}
\usepackage[top=2cm,bottom=2cm,left=2cm,right=2cm]{geometry}
\usepackage{url}
\usepackage{subfigure}
\usepackage{amsfonts,mathrsfs}
\usepackage{amssymb,amsmath}
\usepackage{verbatim}
\usepackage{acronym}
\usepackage{graphicx}
\usepackage{enumerate}

\def\fskip#1{}

\newtheorem{theorem}{Theorem}

\newtheorem{defi}{Definition}

\newtheorem{lemma}{Lemma}

\renewcommand{\theenumi}{\arabic{enumi}}

\def\1{{\bf 1}}

\def\cS{\mathcal{S}}

\newcommand{\cI}{\mathcal{I}}
\newcommand{\cU}{\mathcal{U}}
\newcommand{\cV}{\mathcal{V}}
\newcommand{\cT}{\mathcal{T}}
\newcommand{\cL}{\mathcal{L}}
\newcommand{\cN}{\mathcal{N}}
\newcommand{\cD}{\mathcal{D}}
\newcommand{\bx}{\mathbf{x}}

\begin{document}
\title{On Convergence Rate of Scalar Hegselmann-Krause Dynamics}
\author{\authorblockN{Soheil Mohajer}\thanks{The work of Soheil Mohajer is Supported by The Swiss National Science Foundation under Grant PBELP2-$133369$.}
\authorblockA{Dept.~of Electrical Engineering \& Computer Sciences\\
     University of California at Berkeley, Berkeley, CA 94720\\
     Email: mohajer@eecs.berkeley.edu} \and
     \authorblockN{Behrouz Touri}
  \authorblockA{Coordinated Science Laboratory\\
     University of Illinois, Urbana, IL 61801\\
     Email: touri1@illinois.edu}
}
\maketitle
\begin{abstract}
In this work, we derive a new upper bound on the termination time of the Hegselmann-Krause model for opinion dynamics. Using a novel method, we show that the termination rate of this dynamics happens no longer than $O(n^3)$ which improves the best known upper bound of $O(n^4)$ by a factor of $n$ .
\end{abstract}


\section{Introduction}\label{sec:intro}
One of the key challenges in modeling of social interactions in a complex multi-agent environment is the modeling of opinion dynamics and how agents influence each other's opinion and how new technology and ideas diffuse through such a network.

One of the models that addresses such dynamics is Hegselmann-Krause dynamics which is introduced in \cite{hegselmannk}. Since, the introduction of the model many attempts have been performed to estimate the termination rate of such dynamics and still, up to now, the exact termination time in the case of scalar dynamics is not know.

Because of the simple nature of the Hegselmann-Krause model, this model inspired other engineering applications, especially those from multi-agent systems to apply this dynamics to a given problem. As an example of such applications is distributed rendezvous problem in a robotic network such as a network of space shuttles. In this problem, one may want to gather a set of robots which lack a central coordination to a common place. One approach to handle this problem is to use the Hegselmann-Krause dynamics in such a network. An overview of this method can be found in \cite{Bullo}.

A lower bound can be derived for this dynamics, that shows that this system needs at least $O(n)$ iterations before termination. The best known upper bound is developed in \cite{TouriNedich:Asilomar2011} where it is shown that the termination time happens in at most $O(n^4)$ iterations. This bounds relies on a quadratic time-varying Lyapunov function which is developed and studied in \cite{TouriNedich:CDC2012}.

Conventionally, majority of the previous studies in the domain of distributed averaging, either rely on diameter-type Lyapunov function \cite{hajnal}, \cite{tsitsiklist}, \cite{Tsitsiklis86}, \cite{Jadbabaie03}, \cite{Nedic_cdc07}, or rely on a quadratic Lyapunov functions \cite{Boyd06}, \cite{Olshevsky09a}, \cite{ergodicityb}, \cite{TouriNedich:Approx}, \cite{TouriNedich:Asilomar2011}, \cite{TouriNedich:CDC2012}.

In this study, we chose a different path to derive the new bound. The path to derive our bound is as follows: we consider a linear function of opinions of a well-chosen subset of agents. Unlike the usual Lyapunov functions, this linear function is not necessarily decreasing, however, we can upper bound the number of times it can increase, as well as the amount by which it can increase each time. We show that except an $O(n)$ steps, this function decreases by at least $O(\frac{1}{n})$ in all the remaining times. The fact that the total increment of the function  is bounded above by $O(n^2)$ will conclude our proof.

The structure of this paper is as follows: in Section \ref{sec:problem}, we briefly discuss Hegselmann-Krause dynamics and formally state the problem, with the main result of this work.  In Section~\ref{sec:lyap}, we review some known results, and present and analyze the Lyapunov function, which is the core of this work.  Finally in Section~\ref{sec:conclusion}, we conclude our discussion with some directions for future studies.

\section{Problem Statement}\label{sec:problem}
Consider $n$ agents on the real line. We denote the opinion of agent $i$ at time $t$ by $x_i(t)\in \mathbb{R}$, and so $\bx(0)=(x_1(0),x_2(0),\dots,x_n(0))\in \mathbb{R}^n$, represents  the initial profile of the agents. The opinion profile of the agents evolves over time according to the following dynamics:
\begin{align}\nonumber
x_i(t+1) =\frac{1}{|\cN_i(t)|} \sum_{j\in \cN_i(t)} x_j(t) \qquad \textrm{for $t\geq 1$},
\end{align}
where
\begin{align}\nonumber
\cN_i(t)=\{j\in [n]: |x_j(t) - x_i(t)| \leq \epsilon\}
\end{align}
is the set of neighbors of agent $i$ at time $t$, and we used the short-hand notation $[n]=\{1,2,\dots,n\}$. Note
that the opinion dynamics is completely determined by the confidence values $\epsilon$, and the initial profile $\bx(0)$.

The interesting object in this work is the profile of the system at time $t$, i.e., $\bx(t)=\left( x_1(t),x_2(t),\dots,x_n(t) \right)$, and its dynamic over time. It is shown in \cite{blondel2009krause} that this system converges to a steady-state $\bx^\star$ in finite number of steps, i.e., there is a finite time instance $T$ after which the profile does not evolve any more. The termination time of the dynamics is formally defined as
\[
T=\min\{t: \bx(t)=\bx(t'), \forall t'\geq t\}.
\]

\begin{figure*}[t]
\begin{center}
 	\psfrag{x1}[Bc][Bc]{$\cU(t)=\{x_1(t),\dots,x_{|\cU(t)|}(t)\} $}
 	\psfrag{x2}[Bc][Bc]{$x_{\nu(t)}(t) $}
	\psfrag{xq}[Bc][Bc]{$x_q(t)$}
	\psfrag{xn}[Bc][Bc]{$x_n(t)$}
	\psfrag{e}[Bc][Bc]{$\epsilon$}
	\psfrag{d}[Bc][Bc]{$d(t)$}
	\psfrag{i1}[Bc][Bc]{$\cU$}
	\psfrag{i2}[Bc][Bc]{$\cV$}
	\psfrag{i3}[Bc][Bc]{$\cI_1(t)$}
	\psfrag{i4}[Bc][Bc]{$\cI_{\nu(t)}(t)$}
	\psfrag{i5}[Bc][Bc]{$M(t)$}
	\psfrag{M}[Bc][Bc]{$\underbrace{\phantom{aaaaaaaaaaaaaaaaa}}_{M(t)} $}
\includegraphics[width=0.80\textwidth]{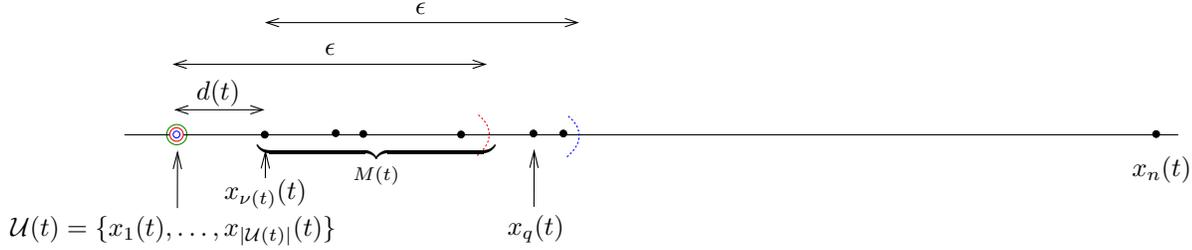}
\end{center}
\caption{Illustration of notation.}
\label{fig:notation}
\end{figure*}

Our goal in this work is to find an upper bound on $T$. To this end, we define a Lyapunov function for the dynamic, and study the variation of  this function over time for the process. This allows us to bound the number of steps its takes for the dynamic to achieve its steady-state. The following theorem states the main result of this work.

\begin{theorem}
For any scalar Hegselmann-Krause dynamics we have
\begin{align*}
T\leq 3n^3+n.
\end{align*}
\end{theorem}

\section{The Lyapunov Function}\label{sec:lyap}
We need to define a some notation and review some known results before presenting and analyzing the Lyapunov function.  Fig.~\ref{fig:notation} is a pictorial representation of the following definitions.

Without loss of generality we may assume that the agents are labelled so that their initial profile is in an increasing order, that is $x_1(0)\leq x_2(0)\leq \cdots \leq x_n(0)$. It is easy to verify that this order is preserved over time, as precisely stated in the following lemma.
\begin{lemma}
The order of the opinion of the agents is preserved over time, i.e., $x_1(t)\leq x_2(t)\leq \cdots \leq x_n(t)$, for all time indices $t$.
\label{lm:order}
\end{lemma}

\begin{defi}
We denote by $\cU(t)$ the set of agents with agreed opinion on the minimum value among all the agents at time $t$, i.e.,
\[
\cU(t)=\left\{i: x_i(t) \leq x_j(t), \  \forall j=1,\dots,n \right\},
\]
\end{defi}

It is clear that once two agents share the same opinion  at time $t$, they share the same neighborhood $\cN(t)$ for the next time instance, and therefore they follow the same dynamics from time $t$ onward. The following lemma presents this property of $\cU(t)$.

\begin{lemma}
The sequence of sets $\{\cU(t)\}$ is  increasing over time, i.e.,
\begin{align}\nonumber
\cU(0) \subseteq \cU(1) \subseteq \cU(2) \subseteq \cdots \subseteq \cU(T).
\end{align}
\label{lm:MostLeft}
\end{lemma}

Let $\nu(t)=|\cU(t)|+1$, which implies $x_{\nu(t)}$ is agent with smallest index whose opinion is located at the second most left position on the real line at time $t$:
\[
\nu(t)=\min\{j: x_j(t) \leq x_i(t) , \ \forall i\in[n]\setminus \cU(t) \}.
\]

Next, we define\footnote{Note that in general there may be more than a single agent collapsed at the most second most left position. However, $\cT(t)$  includes only the one with smallest index.} $\cT(n)=\cU(n) \cup \{x_{\nu(t)}(t)\}$. This is the set of objects who define the Lyapunov function as follows.

\begin{defi}
For the scalar Hegselmann-Krause dynamics, we define the Lyapunov function as
\begin{align}
\cL(t) &= \sum_{i\in \cT(t)} \left( x_n(t) - x_i(t)\right)\nonumber\\
&= |\cU(t)| \Big(x_n(t) - x_1(t)\Big) + \Big(x_n(t) - x_{\nu(t)}(t)\Big)
\end{align}
\end{defi}

We can bound the value of $\cL(t)$ at the initial and steady states.  We also study the behavior of $\cL(t)$ as dynamics evolves to find an upper bound for the termination time of the dynamics.

In the following section, we first focus on dynamics with singular point steady-state, i.e., we assume that at the termination time we have $x_1(T)=x_2(T)=\cdots=x_n(T)$. We show that the termination time for such dynamics satisfies the claim of Theorem~1. Then we generalize this argument to cover systems arbitrary steady-state in Section~\ref{sec:arbit}. 

\subsection{Dynamics with Singular-Opinion Steady-State}\label{sec:singular}
In the following we assume that the profile of the agents at the termination time is single point, and the system never splits into isolated sub-systems. 
First note that
\begin{align}
\cL(0) &< \nu(0) (x_n(0)-x_1(0)) \leq \nu(0) n\epsilon. \label{eq:L0}
\end{align}
We also point out that the value of the function at the steady-state is a finite and non-negative real number. If the system converges to a single opinion (as the focus of this section is), then $\cL(T)= 0$, but $\cL(T)\geq 0$ in general.

It is worth mentioning that  $\cL(t)$ inherently depends on longest distance between the agents in the system, as well as the number agents collapsed at the most left position. While the former decreases over time (Lemma~\ref{lm:order}), the latter ($|\cU(t)|$) is increasing. As a subsequence, the Lyapunov function is not a monotonic function over time. Fig.~\ref{fig:Lyapunov} shows an example of evolution of $\cL(t)$. However, each increment in $\cL(t)$ corresponds to merging of at least a agent to $\cU(t)$, and therefore since this can only happen $n$ times within the entire process, $\cL(t)$ is decreasing unless for at most $n$ time instances. To make it more precise, define\footnote{Recall that $\{\nu(t)\}$ is in increasing sequence.}
\begin{align*}
\cI=\{0\leq t < T: \nu(t+1)>\nu(t) \},\\
\cD=\{0\leq t < T: \nu(t+1)=\nu(t) \}.
\end{align*}

In the following we analyze the variation of $\cL(t)$ for $t\in\cI$
and $t\in\cD$, separately. For $t\in\cI$, although the Lyapunov function
may increase, we show that the total amount of increment for all $t\in\cI$
cannot exceed a certain upper bound. On the other hand, for $t\in\cD$ the
Lyapunov function decreases, and the amount of decrement is lower bounded.
Hence, the number of time instances in $\cD$ is bounded from above.

\paragraph{Variation of $\cL(t)$ for $t\in\cI$}

For  $t\in \cI$, we have $\nu(t+1)\geq \nu(t)+1$. Let  $M(t)= \cN_1(t) \setminus  \cT(t) $, $m(t)=|M(t)|$, and $\tilde{x}(t)=\frac{1}{m}\sum_{i\in M(t)} x_i(t)$. It is clear that $x_i(t)\geq x_{\nu(t)}(t)$ for $i\in M(t)$, and hence $\tilde{x}(t)\geq x_{\nu(t)} (t)$. Therefore,
\begin{align}
x_1(t+1)&= \frac{\sum_{i\in\cN_1(t)} x_i(t)}{|\cN_1(t)|}\nonumber\\
&=\frac{|\cU(t)| x_1(t) + x_{\nu(t)}(t) + \sum_{i\in M(t)} x_i(t)}{ |\cU(t)|+1+|M(t)|}\nonumber\\
&\geq \frac{|\cU(t)| x_1(t) + x_{\nu(t)}(t) }{ |\cU(t)|+1},\nonumber
\end{align}
which implies
\begin{align}
|\cU(t)| x_1(t) + x_{\nu(t)}(t) \leq \nu(t) x_1(t+1) \label{eq:eq1}
\end{align}
Therefore, the difference between $\cL(t+1)$ and $\cL(t)$ for $t\in\cI$ can be upper bounded by
\begin{align}
&\cL(t+1)\hspace{-1pt}-\hspace{-1pt} \cL(t)\hspace{-1pt} \nonumber\\
&\hspace{-2pt}=\hspace{-2pt}\Big[ \hspace{-1pt}\nu(t\hspace{-2pt}+\hspace{-2pt}1)x_n(t\hspace{-2pt}+\hspace{-2pt}1)\hspace{-2pt}-\hspace{-2pt}
\left|\cU(t\hspace{-1pt}+\hspace{-1pt}1)\right|x_1(t\hspace{-2pt}+\hspace{-2pt}1)\hspace{-2pt}-\hspace{-2pt}x_{\nu(t\hspace{-1pt}+\hspace{-1pt}1)}(t\hspace{-2pt}+\hspace{-2pt}1)\hspace{-1pt}\Big]\nonumber\\
&\phantom{\hspace{-2pt}=\hspace{-2pt}}-\Big[ \nu(t)x_n(t) -\hspace{-1pt}\left|\cU(t)\right| x_1(t)-x_{\nu(t)}(t) \Big]\nonumber\\
&\leq \Big[ \nu(t+1) x_n(t+1)-\nu(t) x_n(t)\Big]\nonumber\\
&\phantom{\leq} - \Big[\nu(t+1)-\nu(t) \Big]x_1(t+1)\nonumber\\
&\phantom{\leq} + \Big[ x_1(t+1) -x_{\nu(t+1)} (t+1)\Big]\label{eq:use-of-eq1}\\
&\leq \Big[ \nu(t+1) x_n(t+1)-\nu(t) x_n(t+1)\Big]\nonumber\\
&\phantom{\leq} - \Big[\nu(t+1)-\nu(t) \Big]x_1(t+1)\label{eq:xn-dec}\\
&= \Big[ \nu(t+1)-\nu(t) \Big]\Big[ x_n(t+1)-x_1(t+1)\Big]\nonumber\\
&\leq \Big[ \nu(t+1)-\nu(t) \Big] n\epsilon
\end{align}
where in \eqref{eq:use-of-eq1} we used the inequality in \eqref{eq:eq1}, and \eqref{eq:xn-dec}
is due the facts that $x_n(t)\geq x_n(t+1)$ and $x_1(t+1) \leq x_{\nu(t+1)}(t+1)$.
Hence, we have

\begin{figure}[t]
\begin{center}
 	\psfrag{t1}[Bc][Bc]{$t_1 $}
 	\psfrag{t2}[Bc][Bc]{$t_2$}
	\psfrag{t3}[Bc][Bc]{$t_3$}
	\psfrag{t4}[Bc][Bc]{$t_4$}
	\psfrag{t5}[Bc][Bc]{$t_5$}
	\psfrag{t6}[Bc][Bc]{$t_6$}
	\psfrag{t}[Bc][Bc]{$t$}
	\psfrag{T}[Bc][Bc]{$T$}
	\psfrag{L}[Bc][Bc]{$\cL(t)$}
	\psfrag{L0}[Bc][Bc]{$\cL(0)$}
\includegraphics[width=0.47\textwidth]{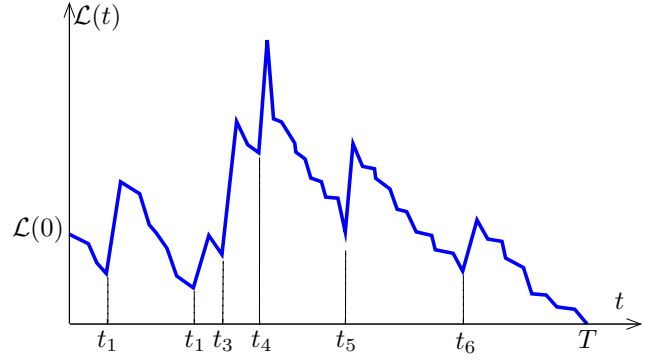}
\end{center}
\caption{Evolution of $\cL(t)$: For any dynamic with singular-opinion steady-state, $\cL(T)=0$. The function is always decreasing except at agents $\cI=\{t_1,t_2,\dots, t_6\}$, with $|\cI|\leq n$.}
\label{fig:Lyapunov}
\end{figure}

\begin{align}
\sum_{t\in \cI} [ \cL(t+1)&- \cL(t) ] \leq \sum_{t\in \cI} \Big[ \nu(t+1)-\nu(t) \Big] n\epsilon\nonumber\\
&= \Big[ \nu(T)-\nu(0) \Big] n\epsilon
= \Big[ n-\nu(0) \Big] n\epsilon\label{eq:1}
\end{align}
Therefore, using \eqref{eq:L0} we get
\begin{align}\nonumber
\cL(0) + \sum_{t\in \cI} \left[ \cL(t+1)- \cL(t) \right] \leq n^2\epsilon.
\end{align}

\paragraph{Variation of $\cL(t)$ for $t\in\cD$}

In the following we study the evolution of $\cL(t)$ for $t\in \cD$. First note the following lemma.
\begin{lemma}\label{lm:non-empty}
For any $t\in \cD$, we have $\cN_{\nu(t)}(t)\setminus \cN_{1}(t) =\emptyset$.
\end{lemma}
\begin{proof}
First, it is clear that $\cN_1(t) \subseteq \cN_{\nu(t)}(t)$, otherwise the distance $x_{\nu(t)}(t)- x_1(t) > \epsilon$, and hence the entire system is split into two isolated systems $\cU(t)$ and $[n]\setminus \cU(t)$, and cannot converge to a singular point.

Now, assume the claim is not true, and $\cN_1(t)=  \cN_{\nu(t)}(t)$.  In this case, we have
\begin{align}\nonumber
x_{\nu(t)}(t+1) &= \frac{1}{|\cN_{\nu(t)}(t)|}\sum_{j\in \cN_{\nu(t)}(t)} x_j(t)\nonumber\\
&= \frac{1}{|\cN_{1}(t)|}\sum_{j\in \cN_{1}(t)} x_j(t)=x_1(t+1),
\end{align}
and therefore, $\nu(t)\in \cU(t+1)$, which yields $\nu(t+1) > \nu(t)$. Note that the latter is in contradiction with  $t\in \cD$.
\end{proof}

Now, let $q \in \cN_{\nu(t)} (t)\setminus \cN_1(t) $.
Denote by $d(t)=x_{\nu(t)}(t)-x_1(t) < \epsilon$ the distance between the most left and second left set of collapsed agents.
Let  $M(t)= \cN_1(t) \setminus  \cT(t) $, $m(t)=|M(t)|$, and $\tilde{x}(t)=\frac{1}{m(t)}\sum_{i\in M(t)} x_i(t)$. Note that $\tilde{X}(t) \geq x_{\nu(t)}(t)$, because all the elements in $M(t)$ have value which are not less than $x_{\nu(t)}$. Hence we have
\begin{align*}
x_1(t&+1)= \frac{\sum_{i\in \cN_1(t)} x_i(t) }{|\cN_1(t)|}\\
&=  \frac{\sum_{i\in \cU(t)} x_i(t) + x_{\nu(t)}(t) + \sum_{i\in M(t)} x_i(t) }{\nu(t)+m(t)}\\
&=  \frac{|\cU(t)|  x_1(t) + (x_1(t)+d(t)) +m(t) \tilde{X}(t) } {\nu(t)+m(t)}\\
&\geq \frac{\cU(t)| x_1(t) + (x_1(t)+d(t)) +m(t)(x_1(t)+d(t)) } {\nu(t)+m(t)}\\
&= x_1(t) + \frac{m(t)+1}{\nu(t)+m(t)} d(t).
\end{align*}
The same variation holds for all the agents in $\cU(t)$, since they all share the same influencing neighbors, and hence,
\[
x_1(t+1)=x_2(t+1)=\cdots =x_{|\cU(t)|}(t+1).
\]

On the other hand,
\begin{align*}
x_{\nu(t)}&(t+1)= \frac{\sum_{i\in \cN_{1}(t)} x_i(t) + \sum_{i\in \cN_{\nu(t)}(t)\setminus \cN_{1}(t)} x_i(t)   }{|\cN_{1}(t)|+ |\cN_{\nu(t)}(t) \setminus \cN_{1}(t)|}\\
&\stackrel{(a)}{\geq} \frac{\sum_{i\in \cN_{1}(t)} x_i(t) + x_q(t)   }{|\cI_{1}(t)|+1}\\
&\stackrel{(b)}{\geq} \frac{|\cU(t)| x_{1}(t) + (m(t)+1)x_{\nu(t)}(t) + (x_{1}(t) +\epsilon  ) }  {\nu(t)+m(t)+1}\\
&= \frac{\nu(t) (x_{\nu(t)}(t)-d(t)) + (m(t)+1)x_{\nu(t)}(t) +  \epsilon   }  {\nu(t)+m(t)+1}\\
&= x_{\nu(t)}(t) + \frac{ \epsilon - \nu(t)d(t)}{\nu(t)+m(t)+1 }.
\end{align*}
where $(a)$ holds since the more agents exist in $\cN_{\nu(t)}(t)\setminus \cN_{1}(t)$, the more moves $x_{\nu(t)}(t)$ towards the right. Moreover, in $(b)$ we used the fact that $x_q(t) \in \cN_{\nu(t)}(t)\setminus \cN_{1}(t)$, which implies $x_q(t)\geq x_1(t)+\epsilon$.

Therefore, assuming $\nu(t+1)=\nu(t)$, we have
\begin{align*}
\cL(t)&- \cL(t+1)\nonumber\\
&= \sum_{i=1}^{\nu(t)} \Big[ (x_n(t)-x_i(t)) - (x_n(t+1)-x_i(t+1)) \Big]\\
&\stackrel{(c)}{\geq} \sum_{i=1}^{\nu(t)} \left[ x_i(t+1)-x_i(t) \right]\\
&\geq |\cU(t)| \frac{m(t)+1}{\nu(t)+m(t)} d(t) + \frac{ \epsilon - \nu(t)d(t)}{\nu(t)+m(t)+1 }\\
&> \frac{ \epsilon }{\nu(t)+m(t)+1 } + \frac{(\nu(t)-1)(m(t)+1)-\nu(t)}{ \nu(t)+m(t)} d(t)
\end{align*}
where $(c)$ holds since $x_n(t)\geq x_n(t+1)$.
Note that the second term above is strictly positive provided that $m(t)>0$. Therefore, for each $t\in\cD$ with $m(t)>0$, the Lyapunov function decreases by at least $\epsilon/n$.

It remains to study the excluded case, where $m(t)=0$. In this case the decrement in the Lyapunov function can be lower bounded by
\begin{align*}
&\cL(t)- \cL(t+1) \geq
 \frac{ \epsilon }{\nu(t)+1 } + \left[ \frac{\nu(t)-1}{\nu(t)}- \frac{\nu(t)}{ \nu(t)+1} \right] d(t)\\
 &\qquad\qquad= \frac{1}{\nu(t)+1} \left[ \epsilon - \frac{d(t)}{\nu(t)}\right]\stackrel{(d)}{\geq} \frac{1}{\nu(t)+1} \left[ \epsilon - \frac{\epsilon}{\nu(t)}\right]\\
 &\qquad\qquad= \frac{\nu(t)-1}{\nu(t)(\nu(t)+1)}\epsilon \stackrel{(e)}{\geq} \frac{1}{3n} \epsilon
\end{align*}
where $(d)$ follows from the fact that $d(t)\leq \epsilon$, and the last inequality in $(e)$ always holds for $3\leq \nu(t)+1\leq n$.

Hence, $\cL(t)$ is decreased by at least $\frac{\epsilon}{3n}$ for each $t\in \cD$. Therefore,
\begin{align*}
0 &\leq  \cL(T) = \cL(0)+ \sum_{t=0}^{T-1} \left[ \cL(t+1)-\cL(t)\right] \nonumber\\
&= \cL(0) + \sum_{t\in \cI} \left[\cL(t+1)-\cL(t)\right] + \sum_{t\in \cD} \left[\cL(t+1)- \cL(t)\right]\\
&\leq  n^2\epsilon - |\cD| \frac{1}{3n} \epsilon,
\end{align*}
which implies $|\cD|\leq 3n^3$. Hence, $T = |\cI| + |\cD| \leq n+ 3n^3=O(n^3)$. Therefore, it takes at most $3n^3+n$ steps to have $\cL(T)\geq 0$, which is the value of the function at the steady state for a singular-point steady-state system. 

\subsection{Dynamics with Arbitrary Steady-State}\label{sec:arbit}
In this part we relax the assumption we made on the steady-state of the system in Section~\ref{sec:singular}, and show that the termination time of the system satisfies the claim on Theorem~1, regardless of the steady-state of the system. To this end, we split the entire process into several phases. The first phase includes the process from its beginning, until the most left agent achieves its steady-state. At this time, the system splits into two isolated sub-systems, namely the most left ones which remain constant for the rest of the process, and the remaining agents. The second phase starts at this time slot, and covers all the slots until the most left agent in the remaining set of agents gets to its steady-state. We can define the next phases similarly, and by adding up the duration of these phases we can bound the termination time of the entire process. 

Recall the argument in Section~\ref{sec:singular}, and recall that it remains unchanged for system with an arbitrary steady-state, except the claim on Lemma~\ref{lm:non-empty}, in which we used the fact that $\cN_1(t) \subseteq \cN_{\nu(t)}(t)$ holds over the entire process. This assumption violates if and only if $x_{\nu(t)}(T_1) -x_1(T_1) >\epsilon$ for some $T_1$. In this case, we have $\cN_1(T_1)=\cU(T_1)$, and the opinion of the agents in $\cU(T_1)$ remains constant for the rest of the time, $t\geq T_1$. In other words, the system splits into a time-invariant part $\cU(T_1)$ and the remaining $[n]\setminus \cU(T_1)$. Hence, we can proceed with the termination time of the remaining systems with $n-|\cU(T_1)|$ agents. 

As stated above, we denote by  $T_1$ the first time instance at which the agents in $\cU(T_1)$ get split from the rest of the system. Now, consider the evolution of the profile from $t=0$ to $t=T_1$. During this interval, the entire argument in Section~\ref{sec:singular} goes through. Define $\cD_1=\cD \cap \{0,1,\dots, T_1-1\}$ and $\cI_1=\cI \cap \{0,1,\dots, T_1-1\}$.  From \eqref{eq:1} we can write   
\begin{align}
\sum_{   t\in \cI_1 } [\cL(t+1)-\cL(t)] \leq \left[ \nu(T_1) -\nu(0) \right] n\epsilon.
\end{align}
Also note that the Lyapunov function decreases by at least $\epsilon/n$ for each $t\in \cI_1$. Therefore, 
\begin{align}
0 &\leq \cL(T_1) \nonumber\\
&= \cL(0) + \sum_{   t\in \cI_1 }[\cL(t+1)-\cL(t)] +  \sum_{   t\in \cD_1 }[\cL(t+1)-\cL(t)] \nonumber\\
  &\leq \nu(T_1)n\epsilon - |\cD_1|\frac{\epsilon}{3n}
\end{align}
which implies $|\cD_1|\leq 3\nu(T_1) n^2$, and hence 
\[T_1 =|\cI_1|+ |\cD_1|\leq \nu(T_1)+ 3\nu(T_1)n^2.
\]

For time slots  $t\geq T_1$, we redefine the Lyapanov function for the new agent set $[n]\setminus \cU(T_1)$, with $n-|\cU(T_1)|$ agents, and proceed until a new set of agents get isolated from the rest at some time slot $T_2$.

We can generalize this argument and define the following splitting-time set
\begin{align*}
\cS 
= \{T_1,T_2,\dots, T_{|\cS|}\},
\end{align*}
which is the set of boundary time indices of the phases defined above.
We  can further split the increasing and decreasing time sets as
\begin{align*}
\cI_k&=\cI \cup [T_{k-1},T_{k}),\\
\cD_k&=\cD \cup [T_{k-1},T_{k}),
\end{align*}
where 
\[ 
[T_{k-1},T_k)=\{T_{k-1},T_{k-1}+1,\dots, T_{k}-1\},
\] 
and $T_0=0$. 
Now, we repeat the above argument for phase the $k$-th phase, and redefine the Lyapanov function for the system with remaining agents in this phase. In order to avoid confusion,  we may use $\cU_k(t)$ and $\cL_k(t)$ to denote the set of agents with opinion at the minimum value among the remaining agents in phase $k$,  and the Lyapanov function in phase $k$, respectively. We also define $\nu_k(t)=|\cU_k(t)|+1$. Note that by definition, the system includes 
\[
n_k= n-\sum_{i=1}^{k-1} |\cU_i (T_i)|
\]
 agents in phase $k$, and this phase takes $T_k- T_{k-1}$ time slots. Hence, we have 
\begin{align}
0&\leq \cL_k(T_k) = \cL_k(T_{k-1}) + \sum_{   t\in \cI_k }[\cL_k(t+1)-\cL_k(t)] \nonumber\\
&\phantom{\leq \cL_k(T_k) = \cL_k(T_{k-1}) } +  \sum_{   t\in \cD_k }[\cL_k(t+1)-\cL_k(t)] \nonumber\\
  &\leq \nu_k(T_k) n_k\epsilon - |\cD_k|\frac{\epsilon}{3 n_k }.\nonumber
\end{align}
Therefore, 
\[
|\cD_k|\leq 3\nu_k(T_k) n_k^2.
\]
 This together with $|\cI_k|\leq \nu_k(T_k)$ implies 
\begin{align*}
T_k-T_{k-1} &\leq \left( 3n_k ^2 +1 \right) \nu_k(T_k) \leq \left( 3n ^2 +1 \right)\nu_k(T_k) .
\end{align*}

Finally, we can upper bound the termination time of the entire system by accumulating all such splitting times. 
\begin{align*}
T &=\sum_{k=1}^{|\cS|} \left(T_k- T_{k-1}\right) \leq \left( 3n^2 +1 \right)  \sum_{k=1}^{|\cS|} \nu_k(T_k) \leq n(3n^2+1)
\end{align*}
where the last inequality follows from the fact that the total number of isolated points during the process  does not exceed $n$. This concludes the desired result for dynamics with arbitrary steady-state. 

\section{Conclusion and Suggestions for Further Studies}\label{sec:conclusion}
We studied the convergence rate of the Hegselmann-Krause model for opinion dynamics and using a novel Lyapunov-type function we proved a new upper bound of $O(n^3)$ for its termination. From the practical perspective, this result is one step towards proving the scalablity of such a dynamics.

  An immediate extension to the current work is to use the technique introduced in this paper to prove a polynomial time upper bound for the termination time of the multi-dimensional Hegselmann-Krause dynamics which has been and is remained to be a very challenging problem.
\bibliographystyle{plain}
\bibliography{Krause}
\end{document}